\documentclass[12pt]{amsart}
\usepackage{latexsym,fancyhdr,amssymb,color,amsmath,amsthm,graphicx,listings,comment}
\usepackage[section]{placeins}
\newtheorem{thm}{Theorem}
\newtheorem{lemma}{Lemma}
\newtheorem{coro}{Corollary}
\let\paragraph\subsection
\title{Fusion inequality for quadratic cohomology} 

\author{Oliver Knill}
\date{June 24, 2024}
\address{Department of Mathematics \\ Harvard University \\ Cambridge, MA, 02138 }
\subjclass{}

\keywords{Quadratic cohomology}

\begin{document}
\maketitle

\begin{abstract}
Classical simplicial cohomology on a simplicial complex $G$ deals with functions on
simplices $x \in G$. Quadratic cohomology \cite{valuation,CohomologyWu}
deals with functions on pairs of simplices $(x,y) \in G \times G$
that intersect. If $K,U$ is a closed-open pair in $G$, we prove here a quadratic version
of the linear fusion inequality \cite{CohomologyOpenSets}.
Additional to the quadratic cohomology of $G$ there are five additional
interaction cohomology groups. Their Betti numbers are computed from functions on
pairs $(x,y)$ of simplices that intersect. Define the Betti vector
$b(X)$ computed from pairs $(x,y) \in X \times X$ with $x \cap y \in X$ a
and $b(X,Y)$ with pairs in $X \times Y$ with $x \cap y \in K$. We prove the fusion inequality
$b(G) \leq b(K)+b(U)+b(K,U)+b(U,K)+b(U,U)$ for cohomology groups
linking all five possible interaction cases. Counting shows
$f(G) = f(K)+f(U) + f(K,U)+f(U,K)+f(U,U)$ for the f-vectors.
Super counting gives Euler-Poincar\'e
$\sum_k (-1)^k f_k(X)=\sum_k (-1)^k b_k(X)$ and $\sum_k (-1)^k f_k(X,Y)=\sum_k (-1)^k b_k(X,Y)$ for
$X,Y \in \{U,K\}$. As in the linear case, also the proof of the quadratic fusion inequality
follows from the fact that the spectra of all the involved Laplacians $L(X),L(X,Y)$
are bounded above by the spectrum of the quadratic Hodge Laplacian $L(G)$ of $G$.
\end{abstract}

\section{In a nutshell}

\paragraph{}
We prove here that if $K$ is a sub-complex of a finite abstract simplicial complex $G$ and $U$ is the
open complement $U=G \setminus K$ \cite{Alexandroff1937,FiniteTopology}, there are besides the quadratic
cohomology of $G$ five quadratic cohomology groups belonging to the five 
{\bf quadratic Hodge Laplacians} $L(X),L(X,Y)$ with $X,Y \in \{U,K\}$. 
They all satisfy spectral inequalities:

\begin{thm}[Spectral inequality]  \label{spectralinequality}
$\lambda_k(L) \leq \lambda_k(L(G))$ \end{thm} 

\paragraph{}
The assumption is that all eigenvalues are ordered in an ascending order and 
that they are {\bf padded left} in comparision with the eigenvalues of $G$. 
This result parallels the linear simplicial cohomology case \cite{HodgeInequality}, 
where $U$ and $K$ can not yet interact and $L$ is one of the Hodge Laplacians 
$L(K),L(U)$ for simplicial cohomology. 

\paragraph{}
In the linear case, the Betti vectors satisfied
the fusion inequality $b(G) \leq b(K) + b(U)$ \cite{CohomologyOpenSets}.
This linear fusion inequality had followed from the spectral inequality and the fact that 
cohomology classes are null-spaces of matrices. Counting gave $f(G)=f(K) + f(U)$ 
for the {\bf $f$-vectors} and the {\bf Euler-Poincar\'e formula} was
$\chi(X)=\sum_{k}(-1)^k f_k(X)=\sum_{k}(-1)^k b_k(X)$ for $X \in \{G,K,U\}$,
seen directly by heat deformation using the {\bf McKean-Singer symmetry}, rephrasing 
that $D_X$ is an isomorphism between even and odd parts of image of
the Laplacian $L_X=D_X^2$, implying ${\rm str}(L^k)=0$ for $k \geq 1$ so that
${\rm str}(\exp(-t L(X))={\rm str}(1_X)=\chi(X)$. 

\paragraph{}
In the quadratic cohomology case, where we look at functions on pairs 
of intersecting simplices, there is the cohomology of $G$ leading to $b(G)$ 
and five interaction cohomologies. They each lead to {\bf Betti vectors}. 
We call them $b(K),b(U)$,$b(K,U),b(U,K)$, and $b(U,U)$. 
Besides pointing out that we have these new cohomologies, we 
give here a relation between them and the cohomology of $G$. We call it 
the {\bf quadratic fusion inequality}.
We could use the heavier notation $b(X,Y,Z)$ with $X,Y,Z \in \{K,U\}$ 
dealing with functions on pairs $(x,y) \in X \times Y$ with $x \cap y \in Z$,
but we prefer to stick to the simpler notation: one reason is that 
both $b(U),b(K)$ deal with internal cohomology of $U$ and $K$ and do 
not involve simplices of the other set,
while $b(K,U),b(U,K),b(U,U)$ involve both sets $U$ and $K$. So, while 
the quadratic Betti vectors $b(U),b(K)$ are {\bf intrinsic} and only depend on
one of the sets $U$ or $K$, the others are not. We prove: 

\begin{thm}[Quadratic fusion inequality]
\label{fusioninequality}
$b(G) \leq b(K)+b(U)+b(K,U) + b(U,K) + b(U,U)$.
\end{thm} 

\paragraph{}
The quadratic Betti vector $b(X)$ belongs to $\{ (x,y), x \cap y \in X \}$,
the vector $b(K,U)$ belongs to all $(x,y) \in K \times U$ with $x \cap y \in K$,
and $b(U,K)$ belongs to all $(x,y) \in U \times K$ with $x \cap y \in K$,
and $b(U,U)$ belongs to all $(x,y) \in U \times U$ with $x \cap y \in K$. 
Note that $x \cap y \in K$ if $x,y \in K$ so that $b(K,K)$ is accounted for
in $b(K)$ already. With the heavier notation, from the eight cases $b(X,Y,Z)$
with $X,Y,Z \in \{K,U\}$ only $b(K)=b(K,K,K),b(U)=b(U,U,U),b(K,U)=b(K,U,K),
b(U,K)=b(U,K,K),b(U,U)=b(U,U,K)$ can occur. 

\paragraph{}
The {\bf quadratic characteristics} are defined as
$w(X)=\sum_{x,y \in X, x \cap y \in X} w(x) w(y)$  (Wu characteristic) and 
$w(X,Y)=\sum_{x \in X,y \in Y, x \cap y \in K} w(x) w(y)$ and
$w(U,U)=\sum_{x \in U,y \in U, x \cap y \in U} w(x) w(y)$, we have 
$w(G)=w(U) +w(K) + w(U,K) + w(K,U) + w(U,U)$ which follows
from the fact that $f$-vectors for quadratic cohomology satisfy
$f(G) =f(K)+f(U)+f(K,U) + f(U,K) +f(U,U)$. 
Any of the cohomologies for $X \in \{ G,U,K,(U,K),(K,U),(U,U) \}$
satisfy the Euler-Poincar\'e formula, following heat deformation with $L_X$
or $L_{X,Y}$ using the McKean-Singer symmetry. 

\paragraph{}
Why is this interesting? If $G$ is finite abstract simplicial complex that 
is a finite $d$-manifold, and $f$ is an arbitrary
function from $G$ to $P=\{0, \dots, k\}$, then the {\bf discrete Sard theorem}
\cite{DiscreteAlgebraicSets}
assures that the open set $U = \{ x \in G, f(x)=P \}$ 
is either empty or a $(d-k)$-manifold in the sense that the graph with vertex set 
$U$ and edge set $\{ (x,y)$ for which $x \subset y$ or $y \subset x \}$ is a 
discrete $(d-k)$-manifold. Since $U$ is open, the complement 
$K=G\setminus U$ is closed and so a sub-simplicial complex. 
All these cohomologies are topological invariants. 

\paragraph{}
For example, if $G$ is a discrete $3$-sphere and 
$f:G \to \{0,1,2\}$ is arbitrary, then $U$ is either empty, 
a {\bf knot} or a {\bf link}, a finite union of
closed disjoint (possibly interlinked) $1$-manifolds in the 3-sphere $G$. 
The simplicial cohomology of $U$ and the quadratic 
cohomology of $U$ are not interesting: they are just circles: $b(U)=(l,l)$, where
$l$ is the number of connected components of $U$. 
The simplicial cohomology of $K$ however can be interesting and leads to knot or 
link invariants. It is well studied in the continuum as it is a
{\bf knot invariant} or a {\bf link invariant}. 
Additional interaction cohomologies that take into account interaction between 
$U$ and the complement $K$ are completely unexplored. 
The inequality shows however that in general, more cohomology classes are created
when splitting up $G$ into $U \cup K$. Unlike in the linear case, we have now
the possibility of particles (harmonic forms) that are functions of $(x,y)$ with 
$x \in U, y \in K$ with $x \cap y \in K$ and also of function on $(x,y)$ with 
$x \in U, y \in U$ with $x \cap y \in K$. 

\paragraph{}
The quadratic case we look at here would generalizes in a straightforward way to 
{\bf higher characteristics}. One starts with $m=1$, the 
linear characteristic which is {\bf Euler characteristic}. 
The second $m=2$ is quadratic characteristic or Wu characteristic going back to Wu in 1959.
\footnote{Historically, multi-linear valuations were considered in
\cite{Wu1959,Gruenbaum1970,valuation} and its cohomology in
\cite{CohomologyWuCharacteristic}. Quadratic cohomology is to Wu characteristic what 
simplicial cohomology is to Euler characteristic.}
In general, we would look for $m$-tuples of points $(x_1,\dots, x_m)$ 
that do simultaneously intersect. 
There are then much more $m$-point interaction cohomologies and $b(G)$ is again 
bounded above by all possible cases of $b(X_1,\dots, X_m)$ with $X_i \in \{U,K\}$ 
and the intersection in $K$. The cases $b(U,U,\dots, U)$ has cases $(x_1, \dots x_m) \in U^m$
with $\bigcap x_i \in U$ which is part of $b_U$ and a new part where $\bigcap x_i \in K$. 
As in the case $m=2$, the cohomology of $b(K,K, \dots, K)$ is part of $b(K)$.
We have to consider the Betti vector $b(U)$ for functions on 
$\{ (x_1,\dots,x_m) \in U^m, \bigcap_i x_i \in U \}$
and $b(U,\dots, U)$ referring to functions on 
$\{ (x_1,\dots,x_m) \in U^m, \bigcap_i x_i \in K \}$, where $b(K,\dots, K)=b(K)$ in the
notation used before. If all $x_i$ are in $U$, then the intersection can be either in $U$
or $K$, but if all $x_i$ are in $K$, then the intersection must be in $K$, not warranting
to distinguish $b(K, \dots, K)$ and $b(K)$. Despite the obvious duality $U \leftrightarrow K$, 
there is an asymmetry in that $x \in U, y \in U$ allows $x\cap y$ to be in $U$ or $K$ while
$x \in K, y \in K$ implies $x \cap y \in K$: technically, a closed set $K$ is a 
$\pi$-system while an open set $U$ is not unless it is $\emptyset$ or $G$. 

\begin{thm}
$b(G) \leq b(U)+b(K) + \sum_{X_i \in \{U,K\}} b(X_1,\dots, X_m)$.
\end{thm} 

There would again be heavier notation $b(X_1, \dots, X_m,X_0)$ dealing
with $(x_1, \dots, x_m) \in X_1 \times \cdots \times X_m$ with $\bigcap_{k=1}^m x_k \in X_0$
and have $b(G) \leq \sum_{X_i \in \{U,K\}} b(X_1,\dots, X_m,X_0)$ taking into account that
some of the cases like $(K, \dots, K,U)$ are empty because $X_0=U$ only is interesting if 
all $X_1= \cdots = X_m=K$. Again, like in the case $m=2$, we have two cases $b(U),b(K)$
which are {\bf intrinsic} while the other cases involve simplices from both $U$ and $K$. 

\paragraph{}
We study here higher order chain complexes in finite geometries. Each of the
situations is defined by a triple $(X,D,R)$, where $X$ is a finite set of $n$ elements,
$D=d+d^*$ is a finite $n \times n$ matrix such that $d^2=0$ and where $R$
is the dimension function compatible with $D$ in the sense that the blocks of $L=D^2$ have
constant dimension. We can call this structure an {\bf abstract delta set} because
every delta set defines such a structure, but where instead of face maps, we just go
directly to the exterior derivative $d$. The advantage of looking at the Dirac setting is
that $D$ can be much more general than coming from face maps. 
It could be a deformed Dirac matrix for example obtained by isospectral deformation 
$D'=[D^+-D^-,D]$ \cite{IsospectralDirac,IsospectralDirac2}, which keeps the spectrum of $D$
invariant but produces $D=d+d^* + B$ leading to new exterior derivatives $d(t)$, 
a deformation which is invisible to the Hodge Laplacian as $D^2(t)=L$ is not affected. 
Since $d(t)$ gets smaller, this produces an expansion of space. In general, also in the 
continuum, there is an inflationary start of expansion. 

\section{A small example}

\paragraph{}
Lets illustrate the quadratic fusion inequality in the case $K_2$: 

\paragraph{}
The linear simplicial cohomology is given by 
$(\left[ \begin{array}{c} \{1\} \\ \{2\} \\ \{1,2\} \end{array} \right],
  \left[\begin{array}{ccc}0&0&-1\\0&0&1\\-1&1&0\\\end{array}\right],
  R=\left[ \begin{array}{c} 0 \\ 0 \\ 1 \end{array} \right] )$. 
Take the open-closed pair $U=\{ \{1,2\}\}$ and $K=\{ \{1\},\{2\} \}$ leading
to the abstract delta set structures 
$(\left[\begin{array}{c} \{1\} \\ \{2\} \end{array} \right],
  \left[\begin{array}{cc} 0&0\\0&0 \\ \end{array}\right],
  \left[\begin{array}{c} 0 \\ 0 \end{array} \right] )$ 
$([ \{1,2\} ],\left[\begin{array}{c} 0 \\ \end{array}\right],[1])$. The Betti vectors
are $b(G)=(1,0),b(U)=(0,1),b(K)=(2,0)$ and the f-vectors are $f(G)=(2,1),f(U)=(0,1),f(K)=(2,0)$. 
The fusion inequality $b(G)<b(K)+b(U)$ is here strict. Merging $U$ and $K$ fuses 
a harmonic $0$ form in $K$ with the $1$-form in $U$.
Betti vectors have been considered since Betti and Poincar\'e. Finite topological
spaces were first looked at by Alexandroff \cite{Alexandroff1937}. For cohomology of open 
sets in finite frame works, see \cite{CohomologyOpenSets}. 

\paragraph{}
If we look at {\bf quadratic cohomology} for $G$, where we have the abstract delta set 
$(X,D,R)=$ 
$$( \left[ \begin{array}{cc}
      \{2\} & \{2\} \\
      \{1\} & \{1\} \\
      \{1,2\} & \{2\} \\
      \{1,2\} & \{1\} \\
      \{2\} & \{1,2\} \\
      \{1\} & \{1,2\} \\
      \{1,2\} & \{1,2\} \\
      \end{array} \right],
     \left[ \begin{array}{ccccccc}
                   0 & 0 & -1 & 0 & 1 & 0 & 0 \\
                   0 & 0 & 0 & 1 & 0 & -1 & 0 \\
                   -1 & 0 & 0 & 0 & 0 & 0 & -1 \\
                   0 & 1 & 0 & 0 & 0 & 0 & 1 \\
                   1 & 0 & 0 & 0 & 0 & 0 & -1 \\
                   0 & -1 & 0 & 0 & 0 & 0 & 1 \\
                   0 & 0 & -1 & 1 & -1 & 1 & 0 \\
           \end{array} \right], 
      \left[ \begin{array}{c}0\\0\\1\\1\\1\\1\\2\end{array}\right]) \; . $$
The Hodge Laplacian $L=D^2=L_0 \oplus L_1 \oplus L_2$ has the Hodge blocks:
$$ L_0=\left[ \begin{array}{cc}
                   2 & 0 \\
                   0 & 2 \\
        \end{array} \right],
   L_1=\left[ \begin{array}{cccc}
                   2 & -1 & 0 & -1 \\
                   -1 & 2 & -1 & 0 \\
                   0 & -1 & 2 & -1 \\
                   -1 & 0 & -1 & 2 \\
       \end{array} \right],
   L_2=\left[ \begin{array}{c}
                   4 \\
       \end{array} \right] $$
with Betti vector $b(G)=(0,1,0)$ and f-vector $f(G)=(2,4,1)$ and
Wu characteristic $w(G)=f_0-f_1+f_2=2-4+1 = b_0-b_1+b_2=0-1+0=-1$. 
The eigenvalues of $L_1$ are $(0,2,2,4)$, the null-space is spanned
by $[1,1,1,1]$. By accident $L_1$ happens to be a Kirchhoff matrix of $C_4$.
If $K_2$ is seen as a 1-manifold with boundary $\delta G$ \footnote{We usually assume
that manifolds with boundary have an interior. The Barycentric refinement of a complete graph
$K_{d+1}$ would be a $d$-manifold with boundary.} 
we have $w(G)=\chi(G)-\chi(\delta(G)) = 1-2=-1$, illustrating that
in general, for manifolds $M$ with boundary $\delta M$, the Wu characteristic is 
$\chi(M)-\chi(\delta M)$. 

\paragraph{}
Now to $U=\{ \{1,2\} \}$, where we have the abstract delta set structure $(X,D,R)=$
$$ (\left[ \begin{array}{cc} \{1,2\} & \{1,2\} \\ \end{array} \right],
   \left[ \begin{array}{c} 0 \\ \end{array} \right],
   \left[ \begin{array}{c} 2 \end{array} \right]) \; . $$ 
with $b(U)=(0,0,1)$ and $f(U)=(0,0,1)$ and $w(U)=1$. \\
For $K=\{ \{1\},\{2\} \}$ we have the abstract delta set structure $(X,D,R) =$
$$ (\left[ \begin{array}{cc} \{2\} & \{2\} \\ \{1\} & \{1\} \\ \end{array} \right],
    \left[ \begin{array}{cc} 0 & 0 \\ 0 & 0 \\ \end{array} \right], 
    \left[ \begin{array}{c} 0 \\ 0 \end{array} \right] ) \;  $$
with $b(K)=(2,0,0)$ and $f(U)=(2,0,0)$ and $w(K)=2$. Obviously the intrinsic
cohomologies of $U$ and $K$ are not yet giving a complete picture. 
The simplices in $U$ and $K$ can interact as we see next. 

\paragraph{}
Now, we turn to the interactions of $K$ with $U$
$$ (\left[ \begin{array}{cc} \{2\} & \{1,2\} \\ \{1\} & \{1,2\} \\ \end{array} \right],
    \left[ \begin{array}{cc} 0 & 0 \\ 0 & 0 \\ \end{array} \right], 
    \left[ \begin{array}{c} 1 \\ 1 \end{array} \right] ) \; .$$
$$ (\left[ \begin{array}{cc} \{1,2\} & \{2\} \\ \{1,2\} & \{1\} \\ \end{array} \right],
    \left[ \begin{array}{cc} 0 & 0 \\ 0 & 0 \\ \end{array} \right], 
    \left[ \begin{array}{c} 1 \\ 1 \end{array} \right] ) \; $$
with $b(K,U)=b(U,K) = (0,2,0)$ and $f(U)=(0,2,0)$ and $w(U,K)=-2$. There is no 
pair $(x,y) \in U^2$ such that $x \cap y \in K$ so that $b(U,U)=(0,0,0)$. 

\paragraph{}
To summarize, we have
$$ \begin{array}{|c|c|c|c|}  \hline
  Case & Betti & f-vector & Characteristic \\ \hline
    U  & (0,0,1) & (0,0,1) & 1  \\
    K  & (2,0,0) & (2,0,0) & 2  \\
 (U,K) & (0,2,0) & (0,2,0) & -2 \\
 (K,U) & (0,2,0) & (0,2,0) & -2 \\
 (U,U) & (0,0,0) & (0,0,0) &  0 \\
    G  & (0,1,0) & (2,4,1) & -1 \\ \hline
\end{array} \; .  $$
The quadratic fusion inequality 
$b(U)+b(K)+b(U,K) + b(K,U) + b(U,U) = (2,3,1) > b(G)=(0,1,0)$
is here strict. The fusion has two $0$-form-$1$-form mergers and 
one $1$-form-$2$-form merger. The difference in the fusion inequality is
is $(1,1,0) + (1,1,0) + (0,1,1)=(2,3,1)$. 

\paragraph{}
We see already in this small example, how the closed ``laboratory" $K$ and the 
``observer space" $U$ are no more strictly separated, even so they partition the
``world" $G$. The ``tunneling" between $K$ and $U$ is described using algebraic 
topology, expressed by cohomology groups.
Unlike for simplicial cohomology which features homotopy invariance, there
is only {\bf topological invariance}. Already the Wu characteristic of 
contractible balls depends on the dimension. [For 
a $d$-ball, the Wu characteristic $w(M)$ is $(-1)^d$ where $d$ is the dimension
and illustrates that $w(M)=\chi(M)-\chi(\delta M)$ in general for discrete manifolds $M$
with boundary $\delta M$ and that for a $d$-ball, the boundary is a $d-1$ sphere
with Euler characteristic $\chi(\delta M) = 1-(-1)^d$.]
If we take a $d$-ball in a $d$-dimensional simplicial complex and replace the interior 
to get an other d-ball without changing the boundary, then the cohomology does not 
change because we can for any positive $k$ add add gauge fields 
(k-forms that are coboundaries) $dg$ to render a cocycle zero
in the interior (without changing the equivalence class) and use the heat flow to 
get back a harmonic form after doing the surgery in the interior. This implements
the {\bf chain homotopy} when doing a local homeomorphic deformation: move the 
field away from the ``surgergy place", do the surgery, then use the heat flow to 
``heal the wound" and get back harmonic forms. 

\paragraph{}
We have just given the argument for the following result:

\begin{thm}
All quadratic cohomology groups $b(X),b(X,Y)$ are topological 
invariants. 
\end{thm}

\paragraph{}
For $G=K_3,K=\{ \{1\} \}$, we can look at the complex 
$G(U,K)= \left[ \begin{array}{cc}
                   \{1\} & \{1,3\} \\
                   \{1\} & \{1,2\} \\
                   \{1\} & \{1,2,3\} \\
                  \end{array} \right]$ and
$D(K,U)=\left[ \begin{array}{ccc}
 0 & 0 & -1 \\
 0 & 0 & 1 \\
 -1 & 1 & 0 \\
\end{array} \right]$ with kernel spanned by $[1,1,0]$. 
For its Barycentric refinement and still $K=\{1\}$ (on the boundary), 
and where we look at functions on 
$X=\left[ \begin{array}{cc}
 \{1\} & \{1,7\} \\
 \{1\} & \{1,5\} \\
 \{1\} & \{1,4\} \\
 \{1\} & \{1,5,7\} \\
 \{1\} & \{1,4,7\} \\
\end{array} \right]$, 
we have 
$D(K,U) = \left[ \begin{array}{ccccc}
 0 & 0 & 0 & -1 & -1 \\
 0 & 0 & 0 & 1 & 0 \\
 0 & 0 & 0 & 0 & 1 \\
 -1 & 1 & 0 & 0 & 0 \\
 -1 & 0 & 1 & 0 & 0 \\
\end{array} \right]$ with kernel spanned by $[1,1,1,0,0]$. 

\section{Quadratic cohomology}

\paragraph{}
{\bf Simplicial cohomology} for a finite abstract simplicial complex $G$ is
part of the spectral theory of the {\bf Hodge Laplacian} $L=D^2$ with 
{\bf Dirac matrix} $D=d+d^*$, where $d$ is the exterior derivative. 
Note that all these matrices $d,D,L$ are $n \times n$ matrices
if $G$ has $n$ elements. The matrix $L$ is a block diagonal matrix 
$L=\oplus_{k=0}^d L_k$.  The kernels of the blocks $L_k$ of $L$ 
are the {\bf $k$-harmonic forms} or {\bf $k$-cohomology} vector spaces. 
In this finite setting, this is linear algebra \cite{Eckmann1944}.
The dimensions $b_k$ are the {\bf Betti numbers}, the components of the 
Betti vector of $G$. 

\paragraph{}
If $K$ is a subcomplex of $G$ 
and $U$ is the open complement, then the
{\bf separated system} $(K,U)$ has a Laplacian $L_{K,U}=L_{K} \oplus L_{U}$ for
which the energies $\lambda_j(L_{K,U})$ are less or equal than $2 \lambda_j(L_G)$
\cite{HodgeInequality} implying that the separated system can not have more harmonic
forms than $G$. It can have more: if $G$ is a closed 2-ball for example and 
$K$ is the boundary 1-sphere then $b_G=(1,0,0), b_K=(1,1,0)$ and $b_U=(0,0,1)$. 
The closed part $K$ carries a trapped harmonic 1-form. 
It is fused with the 2-form present on $U$, if $K,U$ get united to $G$. 

\paragraph{}
A complex $G$ defines a delta set $G=\bigcup_{k=0}^d G_k$. 
The {\bf $f$-vector} $f(G)=(f_0(G),\dots,f_d(G))$ has components 
$f_k(G)=|G_k|$, the number of elements in $G_k$. The {\bf super trace} of an $n \times n$
matrix $L$ \footnote{We write the entries as $L(x,y)$} is defined as 
${\rm str}(L)=\sum_{k=0}^d (-1)^k \sum_{x \in G_k} L(x,x)$. Compare with the usual trace
${\rm tr}(L)=\sum_{k=0}^d \sum_{x \in G_k} L(x,x) = \sum_{x \in G} L(x,x)$. The 
Euler characteristic is $\chi(G)=\sum_{x \in G} w(x)$. The {\bf Euler-Poincar\'e formula} 
$\chi(G)=\sum_k (-1)^k f_k=\sum_k (-1)^k b_k$
follows directly from the {\bf McKean-Singer identity}, stating
that ${\rm str}(\exp(-i t L))=\chi(G)$ for all $t$ which in turn follows from the fact
that the Dirac matrix $D$ gives an isomorphism between even and odd 
{\bf non-harmonic forms}. For $t=0$, the super trace of the heat kernel is the 
combinatorial Euler characteristic, while for $t=\infty$, it is the cohomological 
Euler characteristic. 

\paragraph{}
{\bf Quadratic cohomology} does not build on single simplices $x \in G$ like 
simplicial cohomology but on {\bf pairs of intersecting simplices} 
$(x,y) \in G \times G$. Define $w(x)=(-1)^{{\rm dim}(x)}$.
The quadratic analog of (linear) Euler characteristic $\chi(A)=\sum_{x \in A} w(x)$ is 
the ``Ising type" energy or {\bf Wu characteristic}
$w(A)=\sum_{x,y,x \cap y \in A}  w(x) w(y)$.  It is an example of a multi-linear 
valuation. We also just call it {\bf quadratic characteristic},
an example of {\bf higher characteristic}. \cite{CharacteristicTopologicalinvariants}.

\paragraph{}
The name ``quadratic" is chosen
because it is multi-linear and for $m=2$ a quadratic valuation. 
Similarly as a quadratic form is a multi-linear map,
linear in each argument, the quadratic characteristic 
$w(A,B) = \sum_{x \in A, y \in B, x \cap y \neq \emptyset} w(A) w(B)$ 
(or variants, where we ask the intersection to be in $A$ or $B$) satisfies 
the valuation formula in each of the coordinates, like 
$w(X,U \cup V) = w(X,U) + w(X,V)-W(X,U \cap V)$. 

\paragraph{}
Given an open-closed pair $(U,K)$, 
one can define quadratic cohomology on 
{\bf $k$-forms}. Forms are functions on 
$\Lambda(X,Y)=\{ (x,y) | x \in X, y \in Y, x \cap y \in X \}$ 
and $\Lambda(X)=\{ (x,y) | x \in X, y \in X, x \cap y \in X \}$ and 
$\Lambda(X,Y) = \{ (x,y) | x \in X, y \in Y, x \cap y \neq \emptyset, x \cap y \in K\}$. 
The $k$-forms are the forms on functions with ${\rm dim}(x)+{\rm dim}(y)=k$. 

\paragraph{}
In the case of an open-closed pair, we have five different cohomologies $U,K$,
$(U,K)$, $(K,U)$, $(U,U)$. 
There is no case $(K,K)$ because the intersection of $x \in K, y \in K$ is in $K$. The 
case $(K,K)$ is part of $K$. The case $(U,U)$ looks at pairs such that the intersection 
in in $K$ while $U$ looks at pairs such that the intersection is in $K$. 
We can have a disjoint union  
$$ \Lambda(G) = \Lambda(U) \cup \Lambda(K) \cup \Lambda(K,U) \cup \Lambda(U,K)
                           \cup \Lambda(U,U) \; . $$
\footnote{In the code we part we identify $\Lambda(U,K)$ with $\Lambda(K,U)$ so that
we do not have to probe which of the entries is the closed set. }

\paragraph{}
The {\bf exterior derivative} is inherited from the exterior derivative on products. 
It is $df(x,y) = d_x f(x,y) + w(x) d_y(f,y)$, where $d_x,d_y$ are the usual 
simplicial exterior derivatives but with respect to the first or second coordinate.
If we would look at this derivative on $X \times Y$, the Hodge Laplacians
are the tensor products of the Laplacians on $X$ and $Y$. Even if the 
set-theoretical Cartesian product $X \times Y$ is not a simplicial complex any more, 
we still have a cohomology.  But now, we restrict this exterior derivative to 
pairs $(x,y)$ that intersect. 
We are not aware of such a construction in the continuum. 

\section{Spectral Monotonicity}

\paragraph{}
The proof of the quadratic fusion inequality Theorem~(\ref{fusioninequality}) 
is analog to the linear case. The key is that in each case, the matrix $L$ is the 
square $L=D^2$ of a matrix $D$ which has the property that
a principal sub-matrix of $D$ has intertwined spectrum so that the left padded 
spectral functions of $L$ are monotone. This looks like a technical detail but it 
is important and {\bf at the heart of the entire story}: the matrix $L$ does not have
the property that taking away highest or lowest dimensional simplices produces
principal sub-matrices which themselves come from a geometry.
But the Dirac matrix $D$ does have the property. And since
$D$ has symmetric spectrum with respect to the origin and $D^2=L$, we have also 
monotonocity for $L$.

\paragraph{}
Let us formulate the Cauchy interlace theorem a bit
differently, than usual. The point is that if a principal submatrix $B$ of a self-adjoint
matrix $A$ has the eigenvalues padded left when compared to the eigenvalues of $A$ then
there is a direct comparison between all eigenvalues. This is very general and allows
to talk about monotonicity rather than interlacing. 

\begin{lemma}[Left Padded Monotonicity] 
\label{leftpaddedmonotonicity}
Let $A$ be a symmetric $n \times n$ matrix and $m$ a principal $m \times m$ submatrix,
denote by $\lambda_1 \leq \cdots \leq \lambda_{n}$ the eigenvalues of $A$ and 
$\mu_{n-m} \leq \cdots \leq \mu_n$ the eigenvalues of $B$. Then 
$\mu_k \leq \lambda_k$ for all $n-m \leq k \leq n$. 
\end{lemma} 

\begin{proof}
This follows directly from the interlace theorem and induction
with respect to $m$. Both induction assumption as well as the 
induction steps involve the interlace theorem. 
\end{proof} 

\begin{figure}[!htpb]
\scalebox{1.0}{\includegraphics{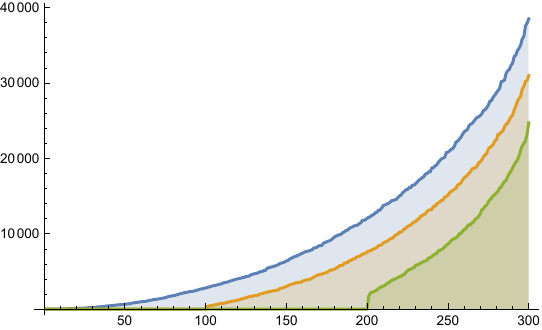}}
\label{rp4}
\caption{
We see the sorted eigenvalues of a random real self-adjoint $300 \times 300$ 
matrix $A_0$, then the eigenvalues of a $200 \times 200$ principal submatrix 
$A_1$ and then the eigenvalues of a $100 \times 100$ principal submatrix
$A_2$ of $A_1$. The eigenvalues are padded left. The figure illustrates 
Lemma~(\ref{leftpaddedmonotonicity}). The code which gave the output is
listed below. 
}
\end{figure} 

\begin{tiny}
\lstset{language=Mathematica} \lstset{frameround=fttt}
\begin{lstlisting}[frame=single]
n=300; m=100; B=Table[20*Random[]-10,{n},{n}]; A0=Transpose[B].B;
A=A0; Do[A1=Transpose[Delete[Transpose[Delete[A,1]],1]]; A=A1,{m}];
A=A1; Do[A2=Transpose[Delete[Transpose[Delete[A,1]],1]]; A=A2,{m}];
T=Eigenvalues;S=Sort;{a,b,c}=PadLeft[{S[T[A0]],S[T[A1]],S[T[A2]]}];
ListPlot[{a,b,c},Joined->True,Filling->Bottom,PlotRange->All];
\end{lstlisting}
\end{tiny} 

\paragraph{}
We can now look at the map $K \to \lambda(K)$ giving for each sub-complex $K$
the spectral function ordered in an ascending way and padded left.
The partial order on sub-simplicial complexes and the partial order on 
spectral functions are compatible:

\begin{coro}
The maps $X \to \lambda(X)$ and
$(X,Y) \to \lambda(X,Y)$ preserve the partial orders in 
the sense that if we remove maximal simplices from closed 
sets, or minimal simplicial from open sets, then the
spectral functions can only get smaller. 
\end{coro}

\paragraph{}
The same holds by iterating the process and take principal 
$(n-k) \times (n-k)$ sub-matrices. Now, if we look at a Dirac 
matrix of a closed set $K$ and take a maximal simplex $x$ away,
then we get monotonicity. The same happens if we take a minimal 
simplex $x$ away from an open set. Note that if we look at 
pairs $(x,y)$ belonging to some pair like $(K,U)$ and we take
a maximal element $x$ away, then several pairs $(x,y_i)$ are removed
from the complex on $(K,U)$. 

\begin{thm}[Spectral monotonicity]
\label{1}
For all $j \leq n$ we have 
$\lambda_j(K) \leq \lambda_j(G)$, \\
$\lambda_j(K) \leq \lambda_j(G)$, \\ 
$\lambda_j(K,U) \leq \lambda_j(G)$, \\ 
$\lambda_j(U,K) \leq \lambda_j(G)$, \\ 
$\lambda_j(U,U) \leq \lambda_j(G)$, \\ 
\end{thm}
\begin{proof}
If we add a locally maximal simplex to a given complex, the spectrum changes monotonically
by interlace. For any vector $u$, 
$\langle u, L u \rangle = \langle u, D^2 u \rangle = \langle Du, Du \rangle = ||Du||^2$
Define $\mathcal{S}_k = \{ V \subset \mathbb{R}^n, {\rm dim}(V)=k \}$
$$ \lambda_k(K) = {\rm min}_{V \in \mathcal{S}_k} {\rm max}_{|u|=1, u \in V} \langle u, L(K) u \rangle
 \leq {\rm min}_{V \in \mathcal{S}_k} {\rm max}_{|u|=1, u \in V} \langle u, L(G) u \rangle 
  = \lambda_k(G) \; . $$
As for the {\bf interlace theorem} applied to $D$ as the Dirac matrix of $K$
is obtained from the Dirac matrix of $L$ by deleting the row and column belonging to the element $x$ which
was added. The eigenvalues of the Dirac matrix $D_K$ of $K$ are now interlacing the eigenvalues of
the Dirac matrix $D_G$ of $G$.
\end{proof}

\paragraph{}
In the quadratic case, taking a way a largest dimensional simplex (facet) $x$ will
affect in general various pairs of simplices $(x,y)$ or $(y,x)$. The effect is 
that the quadratic Dirac matrix of $G \setminus x$ is still a principal sub-matrix. 
We still have spectral monotonicity.

\paragraph{}
To conclude the proof of Theorem~(\ref{fusioninequality}), 
 write down the decoupled Laplacian 
$L(U) \oplus L(K) \oplus L(K,U) \oplus L(U,K) \oplus L(U,U)$ which is block
diagonal and is a $n \times n$ matrix, the same size than $L(G)$. 
Lets call its eigenvalues $\mu_k$. From the spectral inequalities for each block, we know 
$$ 0 \leq \mu_k \leq 5 \lambda_k \; . $$
where $\lambda_k$ are the eigenvalues of the quadratic Hodge Laplacian of $G$. 
Therefore, there are at least as many 0 eigenvalues for the decoupled system
than for $G$, proving the inequality. 

\section{An example}

\paragraph{}
Here is an example with the {\bf Kite complex} $G=K_{1,2,1}$, where we have a 
complex with 2 triangles. We will see what happens if we take one of the 
triangles away. We look at the case $(U,U)$. The Dirac matrix $D(U,U)$ is 
a $14 \times 14$ matrix. 
$$ 
\left[
                  \begin{array}{cccccccccccccc}
                   0 & 0 & 0 & 0 & -1 & 0 & 0 & 0 & -1 & 0 & 0 & 0 & 0 & 0 \\
                   0 & 0 & 0 & 0 & 0 & 0 & -1 & 0 & 0 & -1 & 0 & 0 & 0 & 0 \\
                   0 & 0 & 0 & 0 & 0 & -1 & 0 & 0 & 0 & 0 & -1 & 0 & 0 & 0 \\
                   0 & 0 & 0 & 0 & 0 & 0 & 0 & -1 & 0 & 0 & 0 & -1 & 0 & 0 \\
                   -1 & 0 & 0 & 0 & 0 & 0 & 0 & 0 & 0 & 0 & 0 & 0 & 1 & 0 \\
                   0 & 0 & -1 & 0 & 0 & 0 & 0 & 0 & 0 & 0 & 0 & 0 & 1 & 0 \\
                   0 & -1 & 0 & 0 & 0 & 0 & 0 & 0 & 0 & 0 & 0 & 0 & 0 & 1 \\
                   0 & 0 & 0 & -1 & 0 & 0 & 0 & 0 & 0 & 0 & 0 & 0 & 0 & 1 \\
                   -1 & 0 & 0 & 0 & 0 & 0 & 0 & 0 & 0 & 0 & 0 & 0 & -1 & 0 \\
                   0 & -1 & 0 & 0 & 0 & 0 & 0 & 0 & 0 & 0 & 0 & 0 & 0 & -1 \\
                   0 & 0 & -1 & 0 & 0 & 0 & 0 & 0 & 0 & 0 & 0 & 0 & -1 & 0 \\
                   0 & 0 & 0 & -1 & 0 & 0 & 0 & 0 & 0 & 0 & 0 & 0 & 0 & -1 \\
                   0 & 0 & 0 & 0 & 1 & 1 & 0 & 0 & -1 & 0 & -1 & 0 & 0 & 0 \\
                   0 & 0 & 0 & 0 & 0 & 0 & 1 & 1 & 0 & -1 & 0 & -1 & 0 & 0 \\
                  \end{array}
                  \right] \; . $$
The Eigenvalues of the Laplacian $L(U,U)=D(U,U)^2$ are 
$\{ 4, 4, 4, 4, 2, 2, 2, 2, 2, 2, 2, 2, 0, 0 \}$. 

\paragraph{}
Now lets take away the simplicies which do not involve the triangle
$(1,3,4)$. We have to select the rows and columns in 
$\{ 1, 2, 3, 4, 7, 8, 9, 11 \}$. The Dirac matrix is 
$$ \left[ \begin{array}{cccccccc}
                  0 & 0 & 0 & 0 & 0 & 0 & -1 & 0 \\
                  0 & 0 & 0 & 0 & -1 & 0 & 0 & 0 \\
                  0 & 0 & 0 & 0 & 0 & 0 & 0 & -1 \\
                  0 & 0 & 0 & 0 & 0 & -1 & 0 & 0 \\
                  0 & -1 & 0 & 0 & 0 & 0 & 0 & 0 \\
                  0 & 0 & 0 & -1 & 0 & 0 & 0 & 0 \\
                  -1 & 0 & 0 & 0 & 0 & 0 & 0 & 0 \\
                  0 & 0 & -1 & 0 & 0 & 0 & 0 & 0 \\
                 \end{array} \right] \; . $$
The eigenvalues of $L$ are now 
$\{ 1, 1, 1, 1, 1, 1, 1, 1 \}$. 

\pagebreak

\section{Code}

\begin{tiny}
\lstset{language=Mathematica} \lstset{frameround=fttt}
\begin{lstlisting}[frame=single]
Generate[A_]:=If[A=={},{},Sort[Delete[Union[Sort[Flatten[Map[Subsets,A],1]]],1]]];
L=Length; Whitney[s_]:=Generate[FindClique[s,Infinity,All]]; L2[x_]:=L[x[[1]]]+L[x[[2]]];
(* Linear Cohomology  *)
sig[x_]:=Signature[x]; nu[A_]:=If[A=={},0,L[A]-MatrixRank[A]];
F[G_]:=Module[{l=Map[L,G]},If[G=={},{},Table[Sum[If[l[[j]]==k,1,0],{j,L[l]}],{k,Max[l]}]]];
sig[x_,y_]:=If[SubsetQ[x,y]&&(L[x]==L[y]+1),sig[Prepend[y,Complement[x,y][[1]]]]*sig[x],0];
Dirac[G_]:=Module[{f=F[G],b,d,n=L[G]},b=Prepend[Table[Sum[f[[l]],{l,k}],{k,L[f]}],0];
 d=Table[sig[G[[i]],G[[j]]],{i,n},{j,n}]; {d+Transpose[d],b}];
Hodge[G_]:=Module[{Q,b,H},{Q,b}=Dirac[G];H=Q.Q;Table[Table[H[[b[[k]]+i,b[[k]]+j]],
 {i,b[[k+1]]-b[[k]]},{j,b[[k+1]]-b[[k]]}],{k,L[b]-1}]];
Betti[s_]:=Module[{G},If[GraphQ[s],G=Whitney[s],G=s];Map[nu,Hodge[G]]];
Fvector[A_]:=Delete[BinCounts[Map[Length,A]],1];
Euler[A_]:=Sum[(-1)^(Length[A[[k]]]-1),{k,Length[A]}];
(* Quadratic Cohomology  *)
F2[G_]:=Module[{},If[G=={},{},Table[Sum[If[L2[G[[j]]]==k,1,0],{j,L[G]}],{k,Max[Map[L2,G]]}]]];
ev[L_]:=Sort[Eigenvalues[1.0*L]];
WuComplex[A_,B_,opts___]:=Module[{Q={},x,y,u},
 Do[x=A[[k]];y=B[[l]];u=Intersection[x,y];
    If[((opts=="Open" && Not[x==y] && L[u]>0 && Not[MemberQ[A,u]]) ||
       (Not[opts=="Open"] &&                        MemberQ[A,u])),
    Q=Append[Q,{x,y}]],{k,L[A]},{l,L[B]}];Sort[Q,L2[#1]<L2[#2] &]];
Dirac[G_,H_,opts___]:=Module[{n=L[G],Q,m=L[H],b,d1,d2,h,v,w,l,DD}, Q=WuComplex[G,H,opts];
  n2=L[Q];   f2=F2[Q];   b=Prepend[Table[Sum[f2[[l]],{l,k}],{k,L[f2]}],0];
  D1[{x_,y_}]:=Table[{Sort[Delete[x,k]],y},{k,L[x]}];
  D2[{x_,y_}]:=Table[{x,Sort[Delete[y,k]]},{k,L[y]}];
  d1=Table[0,{n2},{n2}]; Do[v=D1[Q[[m]]]; If[L[v]>0,Do[r=Position[Q,v[[k]]];
    If[r!={},d1[[m,r[[1,1]]]]=(-1)^k],{k,L[v]}]],{m,n2}];
  d2=Table[0,{n2},{n2}]; Do[v=D2[Q[[m]]]; If[L[v]>0, Do[r=Position[Q,v[[k]]];
    If[r!={},d2[[m,r[[1,1]]]]=(-1)^(L[Q[[m,1]]]+k)],{k,L[v]}]],{m,n2}];
  d=d1+d2; DD=d+Transpose[d]; {DD,b}]; 
Beltrami[G_,H_,opts___]:=Module[{Q,P,b},{Q,b}=Dirac[G,H,opts];P=Q.Q];
Hodge[G_,H_,opts___]:=Module[{Q,P,b},{Q,b}=Dirac[G,H,opts];P=Q.Q;
 Table[Table[P[[b[[k]]+i,b[[k]]+j]], {i,b[[k+1]]-b[[k]]},{j,b[[k+1]]-b[[k]]}],{k,2,L[b]-1}]];  
Betti[G_,H_,opts___]:=Map[nu,Hodge[G,H,opts]];
Wu[A_,B_,opts___]:=Sum[x=A[[k]];y=B[[l]];u=Intersection[x,y];
  If[(opts=="Open" && Not[x==y] && L[u]>0 && Not[MemberQ[A,u]]) ||
     (Not[opts=="Open"]    &&                    MemberQ[A,u]),
     (-1)^L2[{x,y}],0],{k,L[A]},{l,L[B]}];
Fvector[A_,B_,opts___]:=Module[{a=F2[WuComplex[A,B,opts]]},Table[a[[k]],{k,2,L[a]}]];

s = CompleteGraph[{1,2,1}];  G = Whitney[s]; K = Generate[{{1,4}}]; U=Complement[G,K]; 
Print["Linear Cohomology"];
{bU,bK,bG}=PadRight[{Betti[U],Betti[K],Betti[G]}];
{fU,fK,fG}=PadRight[{Fvector[U],Fvector[K],Fvector[G]}];
Print[ Grid[{
   {"Case", "Betti","F-vector","Euler"}, {"U", bU,fU,  Euler[U]},
   {"K", bK,fK,  Euler[K]}, {"G", bG,fG,  Euler[G]},
   {"Compare",bU+bK-bG,fU+fK-fG, Euler[U]+Euler[K]-Euler[G]}}]];
Print["Quadratic Cohomology"];
{bU,bK,bKU,bUK,bUU,bG}=PadRight[{Betti[U,U,"Closed"],Betti[K,K,"Closed"], 
   Betti[K,U,"Closed"],Betti[U,K,"Closed"], Betti[U,U,"Open"],Betti[G,G,"Closed"]}]; 
  {fU,fK,fKU,fUK,fUU,fG}=PadRight[{Fvector[U,U,"Closed"],Fvector[K,K,"Closed"], 
  Fvector[K,U,"Closed"],Fvector[U,K,"Closed"], Fvector[U,U,"Open"],  Fvector[G,G,"Closed"]}];
Print[ Grid[{ {"Case","Betti","F-vector","Wu"},{"U",bU,fU,Wu[U,U,"Closed"]},
  {"K",bK,fK,Wu[K,K,"Closed"]},{"UK",bKU,fKU,Wu[K,U,"Closed"]},{"KU",bKU,fKU,Wu[K,U,"Closed"]},
  {"UU",bUU,fUU,Wu[U,U,"Open"]},{"G", bG, fG, Wu[G,G,"Closed"]},
  {"Compare",bU+bK+bKU+bKU+bUU-bG,fU+fK+fKU+fKU+fUU-fG,
  Wu[U,U,"Closed"]+Wu[K,K,"Closed"]+2Wu[K,U,"Closed"]+Wu[U,U,"Open"]-Wu[G,G,"Closed"]}}]];

\end{lstlisting}
\end{tiny}

\paragraph{}
Here is the output of the above lines for simplicial cohomology

$$ \begin{array}{|c|c|c|c|} \hline
 \text{Case} & \text{Betti} & \text{F-vector} & \text{Euler} \\ \hline
 \text{U} & \{0,0,0\} & \{2,4,2\} & 0 \\
 \text{K} & \{1,0,0\} & \{2,1,0\} & 1 \\
 \text{G} & \{1,0,0\} & \{4,5,2\} & 1 \\ \hline
 \text{Compare} & \{0,0,0\} & \{0,0,0\} & 0 \\ \hline 
\end{array} \; . $$

And here the output table for the quadratic cohomology part: 

$$ \begin{array}{|c|c|c|c|} \hline
 \text{Case} & \text{Betti} & \text{F-vector} & \text{Wu} \\ \hline
 \text{U} & \{0,0,0,0,0\} & \{2,8,12,8,2\} & 0 \\
 \text{K} & \{0,1,0,0,0\} & \{2,4,1,0,0\} & -1 \\
 \text{UK} & \{0,0,2,0,0\} & \{0,4,8,2,0\} & 2 \\
 \text{KU} & \{0,0,2,0,0\} & \{0,4,8,2,0\} & 2 \\
 \text{UU} & \{0,0,0,2,0\} & \{0,0,4,8,2\} & -2 \\
 \text{G} & \{0,0,1,0,0\} & \{4,20,33,20,4\} & 1 \\ \hline
 \text{Compare} & \{0,1,3,2,0\} & \{0,0,0,0,0\} & 0 \\  \hline
\end{array}  \; . $$

\begin{figure}[!htpb]
\scalebox{1.0}{\includegraphics{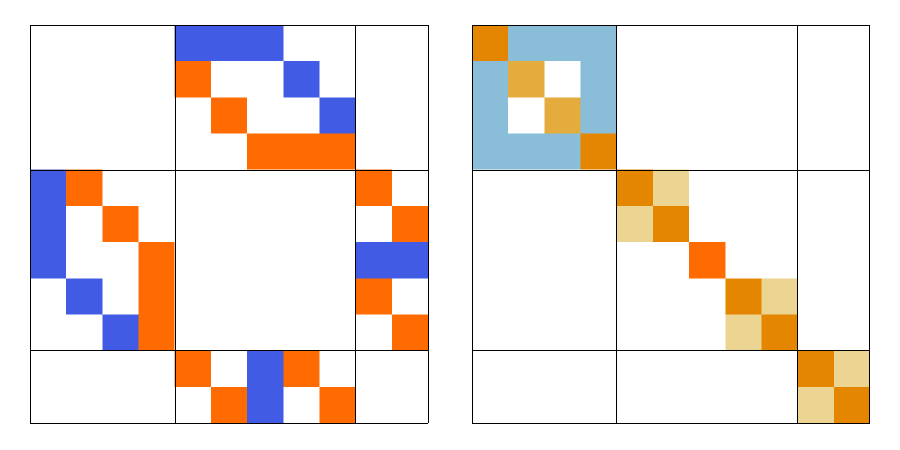}}
\label{rp4}
\caption{
The Dirac matrix $D$ and the Hodge Laplacian $L=D^2$ in the 
linear case for the kite graph $G$. 
The splittings are given by the f-vector $f(G)=(4,5,2)$. 
There are 4 points, 5 edges and 2 triangles in $G$. 
}
\end{figure}

\begin{figure}[!htpb]
\scalebox{1.0}{\includegraphics{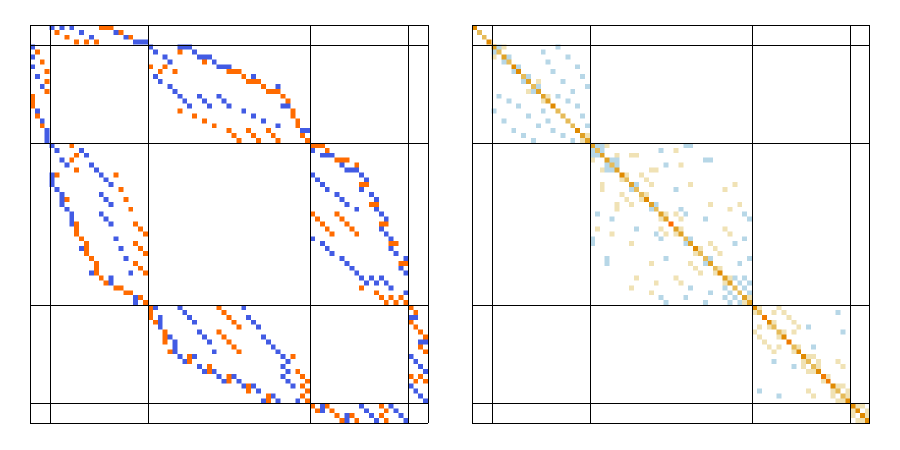}}
\label{rp4}
\caption{
The Dirac matrix $D$ and the Hodge Laplacian $L=D^2$ in the
quadratic case for the kite graph. 
The splittings are given by the f-vector $f(G)=(4,20,33,20,4)$. 
The space of $1$-forms (intersecting  points) is $4$-dimensional,
the space of $2$-forms (intersection of a point with an edges) is $20$-dimensional, 
the space of $3$-forms (intersection of two edges or a triangle-point has dimension $33$),
the space of $4$ forms (intersection of an edge and triangle) is $20$-dimensional,
the space of $5$ forms (intersection of two triangles) is $4$-dimensional.
}
\end{figure}

\bibliographystyle{plain}

\end{document}